\documentclass[12pt]{article}

\usepackage{latexsym,amstext,amsmath,amssymb,amsthm, amsopn}
\usepackage{color}

\newtheorem{thm}{Theorem} \newtheorem{lem}{Lemma}
\newtheorem{df}{Definition}

\DeclareMathOperator{\dist}{dist} \DeclareMathOperator{\diam}{diam}

\def \ah {{$\alpha$--harmonic}}
\newcommand{\appc}[1]{\stackrel{#1}{\approx}}

\def \EE {\mathbb{E}} \def \PP {\mathbb{P}} \def \RR {\mathbb{R}}  \def \Rd {{\RR^d}}      \def \pK {\mathcal{K}} 
 \def \eps {\varepsilon}

      \def \CI {C_1}  \def \CIII
{C} \def \CIV {c_1}

\def \kk {G}
\def \tp{\tilde{p}}
\def \tEE {\tilde{\mathbb{E}}} 
\def \tPP {\tilde{\mathbb{P}}}
\def \tx {\tilde{x}}
\def \tu {\tilde{u}}
\def \ty {\tilde{y}}
\def \tQ {\tilde{Q}}
\def \tG {\tilde G}
\def \GD {G}
\def \GDD {G}
\def \tGD {{\tilde G}}
\def \tGDD {{\tilde G}}
\def \tP {\tilde P}
\def \Ab {\,b\nabla\,}
\def \deltaD {\delta}
\def\deltaDD {\delta}
\def \mk {\kappa}
\def \mgk {\hat{\kappa}}

\def \nn {{\sf n}}
\title{Estimates of the Green function for the fractional Laplacian
  perturbed by gradient\footnote{The research was partially supported by MNiSW. 
2000 Mathematics Subject Classification: 47A55, 60J35, 60J50, 60J75, 47G20.
Key words and phrases:  fractional Laplacian, gradient perturbation, Green function, smooth domain, Kato condition.} } \author{Krzysztof Bogdan, Tomasz
  Jakubowski\footnote{Institute of Mathematics and Computer Science, Wroc{\l}aw University of Technology, Wybrze\.ze Wyspia\'nskiego 27, 50-370 Wroc{\l}aw, Poland,
e-mail: Krzysztof.Bogdan@pwr.wroc.pl, Tomasz.Jakubowski@pwr.wroc.pl}
} \date{\today}

\begin{document}
\maketitle

\begin{abstract}
The Green function of the fractional Laplacian of the differential order bigger than one and the Green function of its gradient perturbations are comparable for bounded smooth multidimensional open sets
if the drift function is in an appropriate Kato class.
\end{abstract}

\section{Introduction}
Perturbations of the Laplace operator $\Delta$ by the first order or gradient
operators $b(x)\cdot \nabla$ were studied by Cranston and
Zhao in \cite{CranZhao}.
They proved for Lipschitz domains that the Green function and the harmonic measure 
 of $\Delta+b(x)\cdot \nabla$
are comparable with those of $\Delta$ under an appropriate Kato condition on the
drift function $b$.
Zhang then showed in \cite{Zhang1} and \cite{Zhang2} 
that the transition density of $\Delta+b\cdot \nabla$
 has Gaussian bounds. 
The results were extended
to more general second order elliptic operators by Liskevich and Zhang (\cite{LZ}),
and to drift {measures} satisfying the Kato condition by Kim and Song (\cite{MR2247841}).

The fractional Laplacian $\Delta^{\alpha/2}$, $0<\alpha<2$, is a primary example of a
non-local generator of a Markovian semigroup.
Perturbations of $\Delta^{\alpha/2}$ 
received much attention recently. 
In particular Schr\"odinger perturbations of $\Delta^{\alpha/2}$ 
were studied by Chen and Song (\cite{MR1473631}, \cite{MR1920109}),
Bogdan and Byczkowski (\cite{MR1671973}, \cite{MR1825645}),  Bogdan, Hansen and Jakubowski (\cite{MR2457489})
 and 
Bogdan, Burdzy and Chen (\cite{MR2006232}). Non-local Schr\"odinger-type
perturbations were considered by Kim and Lee
in \cite{MR2316878}, following earlier papers of Song (\cite{MR1213197}, \cite{MR1353551}). Gradient perturbations of $\Delta^{1/2}$
were studied by Caffarelli  and Vasseur (\cite{MR2680400}) and Kiselev, Nazarov, Volberg (\cite{MR2276260}).
Gradient perturbations of $\Delta^{\alpha/2}$ for $\alpha>1$
were considered by Bogdan and Jakubowski (\cite{MR2283957}) and Jakubowski and Szczypkowski (\cite{2009-TJ-KS-jee}), with focus on sharp estimates of the corresponding transition densities on the whole of $\Rd$. 
In the present paper we estimate the Green function for smooth bounded  subsets of $\Rd$.

Following \cite{MR2283957} we let $\alpha\in (1,2)$.  
We will consider dimensions $d\in\{2,3,\ldots\}$, a nonempty bounded open $C^{1,1}$ set $D\subset \Rd$,
its Green function $G_D$ for $\Delta^{\alpha/2}$, and the Green function $\tilde G_D$
of the operator
$$L = \Delta^{\alpha/2} + b(x)\cdot\nabla\,,$$
where $b$
is a function in Kato class
$\pK_d^{\alpha-1}$ (for details see Section~\ref{sec:p}). 
Our interest in $L$ is motivated by the development of the classical theory of the Laplacian, non-symmetry of $L$ (we have $L^*=\Delta^{\alpha/2} -b(x)\cdot\nabla-{\rm div}\, b$), the fact that  the drift is quite a problematic addition to a jump type process, and by a handful of techniques which already exist for $\Delta^{\alpha/2}$. 

The following estimate, aforementioned in the Abstract,
is an extension to $\Delta^{\alpha/2}$ 
of the results of Cranston and Zhao \cite{CranZhao}.
\begin{thm}\label{Theorem1}
Let $d\geq 2$, $1<\alpha<2$, $b\in \pK_d^{\alpha-1}$, and  let $D\subset \Rd$
be bounded and $C^{1,1}$.
There exists a constant $\CIII =
  \CIII(\alpha,b,D)$ such that for $x,y \in
D$,
  \begin{equation}
    \label{eq:egf}
\CIII^{-1}G_D(x,y) \le \tilde G_D(x,y) \le \CIII G_D(x,y)\,.
  \end{equation}
\end{thm}
Sharp explicit estimates of $G_D$, hence of $\tilde G_D$, exist, see (\ref{GreenEstimates}), and sharp explicit estimates of the corresponding Poisson kernel are given in (\ref{eq:ePk}) below.
Theorem~\ref{Theorem1} is based on the perturbation formula for the Green operators,
$$
  \tG_D = G_D + \tG_D\Ab G_D\,,
$$
where $\Ab \varphi(x) = b(x) \cdot \nabla \varphi(x)$. 
Iterating 
yields {formal} perturbation series, 
$$
  \tG_D = \sum_{n=0}^\infty G_D(\Ab G_D)^n\,.
$$
The structure of the proof of Theorem~\ref{Theorem1} is now as follows.
Section~\ref{sec:p} provides details on the $C^{1,1}$ condition and on transition
densities, Green kernels and harmonic functions of the underlying
Markov processes.
In Section~\ref{s:3} we prove the perturbation formula 
and in Section~\ref{chap:Green} we prove that
the perturbation series 
indeed converge and yield (\ref{eq:egf}) for {small} sets $D$ with bounded {distortion}.  
In the proofs we use estimates for $G_D$ 
(\cite{MR1490808}, \cite{MR1654824}, \cite{MR1991120}) and for the gradient of $G_D$
(\cite{BKN}), 
the boundary Harnack inequality for $\Delta^{\alpha/2}$ (\cite{MR2365478}, \cite{MR1438304}) and the Kato condition (\ref{eq:Kc}) for the drift function $b$.
As a result in Section~\ref{chap:Green} we obtain the Harnack and boundary
Harnack inequalities for nonnegative harmonic functions of $L$ 
in large open sets. 
These are then used in Section~\ref{sec:b}
along with  the perturbation formula and a rough upper bound for $\tG_D$ given in
Lemma~\ref{lem:GLDupper}, to
prove Theorem~\ref{Theorem1} for arbitrary bounded $C^{1,1}$ open sets. 
A number of other auxiliary results are proved in the Appendix.

Concerning the statement of Theorem~\ref{Theorem1}, we note that if the diameter of $D$ is smaller than $r$ and the distortion of $D$ is smaller than $\lambda$, then the constant $C$ in (\ref{eq:egf}) depends only on $d$, $\alpha$, $r$, $\lambda$ and the suprema in the definition of the Kato class $\pK^{\alpha}_d$ (see below).

We observe that an approach similar to ours was recently used
for gradient perturbations of elliptic operators on small sets in \cite{MR2140204} (see also \cite{Zhang1}).
In a wider perspective, Theorem~\ref{Theorem1} is an analogue of the Conditional Gauge Theorem (CGT) in the theory of  Schr\"odinger perturbations, see \cite{MR1671973}, \cite{MR1825645},
\cite{MR1473631}, \cite{MR1920109}, \cite{MR1329992} and \cite{H}.
We should remark here that the distributions of the Markov processes
generated by $\Delta^{\alpha/2}$ and $L$ are {\it not}\/ mutually absolutely continuous locally in time even for (nonzero) constant drift $b$, and any $\alpha\in (0,2)$, see \cite[Theorem 33.1]{MR1739520}. Therefore techniques based on the Girsanov theorem (\cite{CranZhao}) seem unavailable, and we need to proceed via analytic estimates of kernel functions.
Apparently an adaptation of our arguments could be used to give a short analytic proof of CGT (compare \cite{MR1329992}, \cite{MR1671973}, \cite{MR1473631}),  in fact a proof much simpler than that of Theorem~\ref{Theorem1}.
Noteworthy, Green function estimates for Schr\"odinger perturbations 
hold {\it conditionally}\/ under global assumptions of finiteness, e.g.\ gaugeability, existence of (finite) superharmonic functions bounded from below
or smallness of the spectral radius.
Lemma~\ref{lem:GLDupper}, a consequence of the estimates of the
transition densities in \cite{MR2283957}, overrides such assumptions here.
Heuristically, adding drift $b(X_t)dt$ to a stochastic process will not increase its mass on $\Rd$.
This contrasts with a possibly exponential growth of the mass of Feynman-Kac semigroups generated by Schr\"odinger operators.
The drift may, however, change the mass of the process killed off $D$ by trying to push it away from the fatal $D^c$.
This is why  (\ref{eq:egf}) is nontrivial phenomenologically.
Also the symmetry of the semigroup and Green function are lost in the presence of the drift, causing certain technical problems. In this connection we note that $\tilde G_D(y,x)$ may be considered the Green function of $L^*$, and this operator has non-zero Schr\"odinger  part, namely $-{\rm div}\,b$.
Our results apply in particular to 
the Ornstein-Uhlenbeck operator $\Delta^{\alpha/2}+k x \cdot
\nabla$ (for dimensions $d\geq 2$ and $1<\alpha<2$). Here $k$ is a constant. 
We refer to \cite{MR2353039} and \cite{MR2369047}
for estimates of superharmonic functions of this important operator.
We  note that for $k<0$ the drift function $b(x)=k x$ will generally increase the occupation time density (i.e.\ the Green function) 
for sets $D$ containing the origin.

The proof of (\ref{eq:egf}) turned out to be quite difficult to handle, in terms of both the  preliminaries and the auxiliary estimates of the Green function.
Therefore we focused our attention on the more explicit $C^{1,1}$ open sets rather than Lipschitz open sets.  
We hope that our approach may now be adapted in the Lipschitz case.
Here the sensitive elements are Lemma~\ref{lem:UnifInt} and (\ref{eq:alphag1}).

A few additional comments on possible extensions of the results are due.
If $d=1<\alpha$, then 
then the right hand side of (\ref{eq:GradEstimGreen}) below will no longer be integrable.  This explains our restriction to $d\ge 2$.
We however conjecture that Theorem~\ref{Theorem1} does extend to $d=1$. This case is interesting even for the sake of the one-dimensional Ornstein-Uhlenbeck process.
One may wonder if  (\ref{eq:egf}) holds for $\alpha=1$,
but we certainly know that (\ref{eq:egf})  fails for $\alpha\in (0,1)$. Indeed, if $0<\alpha<1$ then the expected exit times from balls, to wit, $\int_{B(x_0,r)} G_{B(x_0,r)}(x,y)dy$, are generally incomparable for $\Delta^{\alpha/2}$ and the Ornstein-Uhlenbeck operator $\Delta^{\alpha/2}+k x \cdot\nabla$ (see \cite{MR2353039}), so the Green functions are not comparable either. Heuristically, a (first-order) gradient perturbation is infinitesimally small with respect to $\Delta^{\alpha/2}$ only if $\alpha>1$. This explains the restriction $1<\alpha<2$ in \cite{MR2283957}, \cite{2009-TJ-KS-jee} and the present paper.
We remark that the existence of ratios and Martin representation of nonnegative harmonic functions of $L$ may likely be obtained with the results and toolbox presented in this paper and  \cite{MR2365478}.
We also note that a similar approach should apply to 
{additive} perturbations of $\Delta^{\alpha/2}$ by non-local L\'evy-type operators
(compare \cite{MR2417435}),
provided
(\ref{eq:GradEstimGreen})  can be generalized.  It also seems possible and interesting to study drift perturbations of more general semigroups subordinated to the Gaussian semigroup (\cite{MR2598208}).

\section{Preliminaries}\label{sec:p}
In what follows, $\Rd$ denotes the Euclidean space of dimension $d\ge
2$, $dy$ stands for the Lebesgue measure on $\Rd$, and we let
$$1<\alpha<2.$$ Without further mention we will only consider Borelian sets, measures and functions in $\Rd$.
By $x\cdot y$ we denote the Euclidean scalar product of $x,y\in \Rd$.
We let $B(x,r)=\{y\in \Rd:
|x-y|<r\}$.
For $D\subset \Rd$ we denote 
$$\delta_D(x) =\dist(x,D^c)\,,$$ 
the distance to the complement of $D$. 
\begin{df}\label{def:C11}
 Nonempty open $D\subset \Rd$ is of class $C^{1,1}$ at scale $r>0$ 
if for
  every $Q\in \partial D$ there are balls 
$B(x',r)\subset D$ and $B(x'',r)\subset D^c$ tangent at $Q$. 
\end{df}
Thus, $B(x\rq{},r)$ and $B(x\rq{}\rq{},r)$ are the {\it inner}\/ an {\it outer}\/ balls tangent at $Q$, respectively. If $D$ is $C^{1,1}$ at some unspecified scale (hence also at all smaller scales), 
then we simply say $D$ is $C^{1,1}$.  
The {\it localization radius}, $$r_0=r_0(D)=\sup\{r: D \mbox{ is } C^{1,1} \mbox{ at scale } r\},$$
refers to the local geometry of $D$, while the {\it diameter},
$${\rm diam}(D)=\sup\{|x-y|:\;x,y\in D\}\,,$$ refers to the global geometry of $D$.
The ratio ${\rm diam}(D)/r_0(D)\geq 2$ will be called the {\it distortion}\/ of $D$. 
We can localize each $C^{1,1}$ open set as follows.
\begin{lem}\label{l:loc}
There exists $\kappa>0$ such that if $D$ is $C^{1,1}$ at scale $r$ and $Q\in \partial D$, then there is a $C^{1,1}$ domain $F\subset D$ with $r_0(F)>\kappa r$, ${\rm diam }(F)< 2r$ and 
\begin{equation}\label{eq:loc}
D\cap B(Q,r/4)=F\cap B(Q,r/4)\,.
\end{equation}
\end{lem}
\noindent
We will write $F=F(z,r)$, and we note that the distortion of $F$ is at most $2/\kappa$, an absolute constant. The proof of Lemma~\ref{l:loc} is given in the Appendix. 
$$\mbox{In what follows $D$  will be a nonempty bounded $C^{1,1}$ open set in $\Rd$.}$$
We note that such $D$ may be disconnected but then it may only have a finite number of connected components, at a positive distance from each other.

We will now give a brief review of the potential theory of the fractional Laplacian, and of the fractional Laplacian perturbed by gradient operators. The former case is well known (\cite{Landkof}, \cite{MR850715}, \cite{MR2569321}, \cite{MR1671973}, \cite{MR2365478}). The latter case is similar but 
we feel it calls for more details,
and they are given in the Appendix.

Let $\mathcal{A}_{d,\gamma}= \Gamma\big((d-\gamma)/2\big)/ (2^\gamma
\pi^{d/2}|\Gamma(\gamma/2)|)$ and
$$  \nu(y)=\mathcal{A}_{d,-\alpha}|y|^{-d-\alpha}\,,\quad y\in \Rd\,.
$$
The coefficient $\mathcal{A}_{d,-\alpha}$ is so chosen that
\begin{equation}
  \label{eq:trf}
  \int_{\Rd} \left[1-\cos(\xi\cdot y)\right]\nu(y)dy=|\xi|^\alpha\,,\quad
  \xi\in \Rd\,.
\end{equation}
For (smooth compactly supported) $\phi\in C^\infty_c(\Rd)$,
the fractional Laplacian is
\begin{equation}\label{eq;deful}
  \Delta^{\alpha/2}\phi(x) = 
  \lim_{\varepsilon \downarrow 0}\int_{|y|>\varepsilon}
  \left[\phi(x+y)-\phi(x)\right]\nu(y)dy\,,
  \quad
  x\in \Rd
\end{equation}
(see \cite{MR1671973, MR2569321} for a broader setup).
If $x\not \in {\rm supp}\, \phi$ then
\begin{equation}\label{eq;defulpn}
  \Delta^{\alpha/2}\phi(x) = 
  \int_{\Rd}
  \phi(y) \nu(y-x)dy\,.
\end{equation}
If $r>0$ and $\phi_r(x)=\phi(r x)$ then
\begin{equation}
  \label{eq:scfr}
  \Delta^{\alpha/2}\phi_r(x)=r^\alpha \Delta^{\alpha/2}\phi(rx)\,,\quad
  x\in \Rd\,.  
\end{equation}
In this respect, $\Delta^{\alpha/2}$ behaves like differentiation of order $\alpha$.
We let $p_t$ be the smooth real-valued function on $\Rd$
with Fourier transform
\begin{equation}
  \label{eq:dpt}
  \int_ \Rd p_t(x)e^{ix\cdot\xi}\,dx=e^{-t|\xi|^\alpha}\,,\quad t>0\,,\;\xi\in
  \Rd\,.
\end{equation}
According to 
(\ref{eq:trf}) 
and the L\'evy-Khinchine formula,
$\{p_t\}$ is a probabilistic convolution semigroup with L\'evy measure
$\nu(y)dy$, see \cite{MR1739520}, \cite{MR2320691} or
\cite{MR2569321}. Let
$$
p(t,x,y)=p_t(y-x)\,.
$$
Using (\ref{eq:dpt}) one proves that $p$ is the heat kernel of the fractional Laplacian:
\begin{equation}
  \label{eq:fsol}
  \int\limits_{s}^\infty\int\limits_{\Rd}
  p(u-s,x,z)\left[
    \partial_u\phi(u,z)+\Delta^{\alpha/2}_z \phi(u,z)\right]\,dzdu
  = -\phi(s,x)\,,
\end{equation}
where $s\in \RR$, $x\in \Rd$ and $\phi\in C^\infty_c(\RR\times \Rd)$.

We consider the time-homogeneous transition probability 
$$
(t,x,A)\mapsto\int_A p(t, x,y)dy\,,\quad t>0\,,\;x\in \Rd\,,\; A\subset
\Rd\,.
$$
By Kolmogorov's and Dinkin-Kinney's theorems the transition
probability defines in the usual way Markov probability measures
$\{\PP^x,\,x\in \Rd\}$ on the space $\Omega$ of the
right-continuous and left-limited functions $\omega :[0,\infty)\to \Rd$.
We let $\EE^x$ be the corresponding integrations.
We will denote by $X=\{X_t\}_{t\geq 0}$ the canonical process on $\Omega$,
$X_t(\omega)=\omega(t)$.
In particular, according to (\ref{eq:dpt}),
\begin{equation}\label{character}
  \EE^0 e^{iX_t\cdot \xi} = e^{-t|\xi|^\alpha}, \qquad \qquad \xi \in \RR^d, \;t \ge 0\,.
\end{equation}
In fact, $(X,\PP^0)$ is a L\'evy process in $\RR^d$
with zero Gaussian part and drift, and with $\nu(y)dy$ as the L\'evy
measure \cite{MR1739520}.
It follows from (\ref{eq:dpt}) that
\begin{equation}
  \label{eq:sca}
  p_t(x)=t^{-d/\alpha}p_1(t^{-1/ \alpha}x)\,,\quad t>0\,,\;x\in \Rd\,.
\end{equation}
It is well-known that $p_1(x)\appc{C}1\land |x|^{-d-\alpha}$, hence
\begin{equation}\label{eq:oppt}
p_t(x)\appc{C}t^{-d/\alpha}\land \frac{t}{|x|^{d+\alpha}}\,,\quad t>0,x\in \Rd\,.
\end{equation}
Symbol $\appc{C}$ means that either
ratio of the sides is bounded by $C\in (0,\infty)$, and
$C$ does not depend on the variables shown, here $t$ and $x$.
We will write mere $\approx$ if $C$ is unimportant or understood.
Constants will usually be denoted with generic $C$ (in statements) or $c$ (in proofs), and we will occasionally enumerate them for convenience of referencing. 
As usual, $a \land b =\min(a,b)$ and $a\vee b = \max(a,b)$. 
In what follows we will often use the identity
\begin{equation}\label{eq:max}
ab=(a\land b)(a\vee b)\,.
\end{equation}
In view of (\ref{eq:sca}) and the fact that each $p_t$ is a radial
function, $X$ is called the isotropic $\alpha$-stable L\'evy process
(see \cite{MR1739520}, \cite{MR2320691} for a discussion of general stable L\'evy processes).  
We introduce the Riesz potential kernel (for $d>\alpha$),
\begin{equation}\label{eq:pkkk}
 {\cal A}_{d,\alpha} |x|^{\alpha-d}=\int_0^\infty p_t(x)dt\,, \quad x\in \Rd\,.
\end{equation}
This is infinite if $x=0$, see (\ref{eq:oppt}).

To study $\Delta^{\alpha/2}$ with Dirichlet conditions we will consider the {\it time of the first exit}\/ of the (canonical) process from $D$,
$$\tau_D=\inf\{t>0: \, X_t\notin D\}\,.$$
We let $\omega^x_D(B)=\PP^x(X_{\tau_D}\in B)$, the $\alpha$-harmonic
measure of $D$ (\cite{MR850715}, \cite{MR0126885}, \cite{Landkof}).
The joint distribution of $(\tau_D, X_{\tau_D})$ 
defines the transition density of the process {\it killed}\/ when leaving $D$
(\cite{MR0079377}, \cite{MR0264757}, \cite{MR1329992}):
$$
p_D(t,x,y)=p(t,x,y)-\EE^x[\tau_D<t;\, p(t-\tau_D, X_{\tau_D},y)],\quad t>0
,\,x,y\in \Rd \,.
$$
By Blumenthal's 0-1 law, radial symmetry of $p_t$ 
and $C^{1,1}$ geometry of the boundary of $\partial D$, we have 
$\PP^x(\tau_D=0)=1$ for every $x\in D^c$. 
In particular, $p_D(t,x,y)=0$ if $x\in D^c$ or $y\in D^c$.
By the strong Markov property,
$$
\EE^x[t<\tau_D;\, f(X_t)]=
\int_\Rd f(y)p_D(t,x,y)dy
\,,
\quad t>0\,,\; x\in\Rd\,,
$$
for functions $f\geq 0$.
The Chapman-Kolmogorov equations hold for $p_D$,
$$
\int_\Rd p_D(s,x,z)p_D(t,z,y)dz=p_D(s+t,x,y)\,,\quad s,t>0 ,\,
x,y\in \Rd\,.
$$
Also, $p_D$ is jointly continuous when $t\neq 0$, and we have
\begin{equation}\label{eq:gg}
  0\leq p_D(t,x,y)=p_D(t,y,x)\leq p(t,x,y)\,.
\end{equation}
In particular,
\begin{equation}
  \label{eq:9.5}
  \int_\Rd p_D(t,x,y)dy\leq 1\,.
\end{equation}
For $s\in \RR$, $x\in \Rd$, and $\phi\in C^\infty_c(\RR\times D)$, we
have (compare (\ref{eq:fsol}))
\begin{equation}
  \label{eq:fsolD}
  \int\limits_{s}^\infty\int\limits_{\Rd}
  p_D(u-s,x,z)\left[
    \partial_u\phi(u,z)+\Delta^{\alpha/2}_z \phi(u,z)\right]\,dzdu
  = -\phi(s,x)\,,
\end{equation}
which justifies calling $p_D$ the heat
kernel of the (Dirichlet) fractional Laplacian {on} $D$. 
We define
\begin{equation}\label{eq:12.5}
G_D(x,y)=\int_0^\infty p_D(t,x,y)dt, \quad x,y\in \Rd\,.
\end{equation}
It follows that $G_D(x,y)$ is symmetric and lower semi-continuous,
and
\begin{equation}\label{eq:GreenFormula}
G_D(x,y) + \int_{D^c} \mathcal{A}_{d,\alpha} |y-z|^{\alpha-d} \omega_D^x(dz)= \mathcal{A}_{d,\alpha} |x-y|^{\alpha-d}\,.
\end{equation}
The Green operator of $\Delta^{\alpha/2}$ for $D$ is
$$
G_Df(x) = \EE^x \int_0^{\tau_D}f(X_t)dt=\int_{\Rd}G_D(x,y)f(y)dy, \quad
x \in \RR^d\,,
$$
and we have 
\begin{equation}\label{eq:fg}
G_D (\Delta^{\alpha/2}\phi)(x)=-\phi(x)\,,\quad x\in \Rd\,,\quad \phi \in C^\infty_c(D)\,.
\end{equation}
A result of Ikeda and Watanabe \cite{MR0142153} asserts that for $x\in
D$ the $P^x$-distribution of $(\tau_D, X_{{\tau_D}-},X_{\tau_D})$
restricted to $X_{{\tau_D}-}\neq X_{\tau_D}$ is given by the density
function
\begin{equation}
  \label{eq:IWf}
  (s,u,z)\mapsto p_D(s,x,u)\nu(z-u)\,.
\end{equation}
The $C^{1,1}$ geometry of $D$
implies that $P^x(X_{{\tau_D}-}\neq X_{\tau_D})=1$ for $x\in D$ (\cite{MR1438304}).
By (\ref{eq:12.5}), (\ref{eq:IWf}) and Tonelli's theorem the $P^x$-distribution of
$X_{\tau_D}$ has a density function, called the Poisson kernel and defined as
\begin{equation}
  \label{eq:Pk}
  P_D(x,z)=\int_D G_D(x,y)\nu(z-y)dy\,.
\end{equation}
The Green function and Poisson kernel of the ball
are known explicitly:
\begin{equation}\label{wfg}
  G_{B(x_0,r)}(x,v)=
{\cal B}_{d,\alpha}\,
|x-v|^{\alpha-d}\int_{0}^{w}\frac{s^{\alpha/2-1}}
  {(s+1)^{d/2}}\,ds\,,
\end{equation}
\begin{equation} \label{eq:poisson:ball}
P_{B(x_0,r)}(x, y) =
{\cal C}_{d, \alpha}
\left[\frac{r^2 - |x-x_0|^2}{|y-x_0|^2 -
      r^2}\right]^{\alpha / 2} |x - y|^{-d} \,, 
\end{equation}
where ${\cal  B}_{d,\alpha}=\Gamma(d/2)/(2^{\alpha}\pi^{d/2}[\Gamma(\alpha/2)]^{2})$,
${\cal C}_{d,\alpha}=\Gamma(d/2)\pi^{-1-d/2}\sin (\pi \alpha/2)$,
$$
w=(r^2-|x-x_0|^{2})(r^2-|v-x_0|^{2})/|x-v|^{2}\,,
$$
$|x-x_0|<r$, $|v-x_0|<r$, and  $|y-x_0|\geq r$;
see \cite{MR0126885}, \cite{bib:Rm} or \cite{Landkof}.

The next estimate was proved by Kulczycki \cite{MR1490808} and Chen and
Song \cite{MR1654824},
\begin{eqnarray}
  &&G_D(x,y)\, 
  \appc{C}\;  |x-y|^{\alpha-d}\left(\frac{\delta_D(x)^{\alpha/2}\delta_D(y)^{\alpha/2}}{|x-y|^\alpha} \land 1\right)
\label{GreenEstimates}\\
  &&\approx\;  |x-y|^{\alpha-d}
\frac{\delta_D(x)^{\alpha/2}\delta_D(y)^{\alpha/2}}{[\delta_D(x)\vee |x-y| \vee \delta_D(y)]^\alpha}
\,, \qquad x,y \in D\,.
\label{GreenEstimatesTJ}
\end{eqnarray}
The reader may check equivalence of (\ref{GreenEstimates}) and (\ref{GreenEstimatesTJ}) by first considering the case $\delta_D(x)\appc{3}\delta_D(y)$.
We like to remark that (\ref{GreenEstimatesTJ}) may be also regarded a direct consequence of the approximate factorization of the Green function of Lipschitz open sets,  see \cite[Theorem 21]{MR1991120}.
It is well known that $C=C(d,\alpha,\lambda)$ in (\ref{GreenEstimates}) and (\ref{GreenEstimatesTJ}), 
if  ${\rm diam}(D)/r_0(D)\leq \lambda$, i.e. $C$ may be so selected to depend only on $d$, $\alpha$ and (an upper bound for) the distortion of $D$. This follows from the proofs of \cite{MR1490808} and \cite{MR1654824} and is  explicitly stated in \cite{MR1991120}, see also \cite{MR1490808}.

We will consider a nonnegative function $u$ on $\Rd$, and an open set $U \subset \Rd$.
$u$ is called $\alpha$-harmonic on $U$ if for each open bounded  
$V\subset\overline{V} \subset U$,
$$
u(x) = \EE^xu(X_{\tau_V}), \quad x \in V.
$$
We say that $u$ is
regular $\alpha$-harmonic on $U$ if also
$$
u(x) = \EE^xu(X_{\tau_U})
\,, \quad x \in U\,.
$$
Here we assume absolute integrability of the expectations, and $\EE^xu(X_{\tau_U})$ is understood as
$\EE^x[\tau_U<\infty;\, u(X_{\tau_U})]$.
For instance $x\mapsto G_D(x,y)$ is
$\alpha$-harmonic in $D \setminus \{y\}$. 
In fact, by the strong Markov property, $G_D(x,y)=G_V(x,y)+\EE^x G_D(X_{\tau_V},y)$ for every open $V\subset U$, and  $G_V(x,y)=0$ if $\dist(y,V) >0$  (see, e.g.,
\cite{MR2365478}).

The following two results can be found in \cite{MR1438304}, see also \cite{MR2365478}.
\begin{lem}[Harnack inequality]
  \label{HI}
  Let $x,y \in \Rd$, $s>0$ and $k\in \mathbb{N}$ satisfy $|x-y| \le
  2^ks$. Let  function $u$ be nonnegative in $\Rd$ and {\ah}
  in $B(x,s) \cup B(y,s)$. There is $C = C(d, \alpha)$ such that
  \begin{equation}
    \label{HP}
    C^{-1}2^{-k(d+\alpha)}u(x) \le u(y) \le C 2^{k(d+\alpha)}u(x).
  \end{equation}
\end{lem}

\begin{lem}
  \label {BHP}
  Let $Z \in \partial{D}$ and $r \in (0,r_0]$, $0<p<1$. Assume that
  functions $u$, $v$ are nonnegative in $\Rd$ and regular \ah{} and non-zero
  in $D\cap B(x_0,r)$. If $u$ and $v$ vanish on $D^c \cap B(x_0,r)$ then
  \begin{equation}
    \label{BHPE}
    C^{-1}\frac{u(x)}{v(x)} \le \frac{u(y)}{v(y)} \le C\frac{u(x)}{v(x)},
  \end{equation}
  for $x,y \in D \cap B(x_0,pr)$. Here $C= C(d, \alpha,p)$.
\end{lem}
We like to remark that the boundary Harnack inequality (Lemma~\ref{BHP}) in fact holds for general open sets and is equivalent to an approximate factorization of the Poisson kernel of general open sets, see \cite{MR2365478}.
We encourage the reader to factorize $P_{B(0,1)}(x,y)$ when $x,y$ are not too close to each other.
In passing we also note that an approximate factorization of $p_D(t,x,y)$ for Lipschitz domains is given in \cite{2010-KB-TG-MR-aop}.
Concluding this part of our preliminary discussion we refer the reader to \cite{MR2569321}, \cite{MR2365478} for more details and
references.

 We note that $\alpha$-harmonic functions are smooth where $\alpha$-harmonic; use (\ref{eq:poisson:ball}) or see \cite{MR1671973}.
The following gradient estimate is given in \cite[Lemma 3.2]{BKN}.
\begin{lem}\label{GradEstim}
Let $U$ be an arbitrary open set in $\Rd$. For every nonnegative
function $u$ on $\Rd$ which is \ah{} in $U$ we have
\begin{equation}
  \label{eq:5}
|\nabla u(x)| \le d \frac{u(x)}{\delta_U(x)}\,, \quad x\in U\,.
\end{equation}
\end{lem}
Since $G_U(\cdot,y)$ is $\alpha$-harmonic in $U \setminus\{y\}$, for
every $y \in U$ we obtain
\begin{equation}\label{eq:GradEstimGreen}
|\nabla_x G_U(x,y)| \le d \frac{G_U(x,y)}{\delta_U(x) \land |x-y|}\,,
\quad x,y \in U,\,\; x\neq y\,.
\end{equation}
We note in passing that a reverse inequality
holds locally at the boundary of
Lipschitz domains, with constant depending on the Lipschitz character
of $D$
(\cite[Lemma 4.5]{BKN}). 
In this sense (\ref{eq:5}) and (\ref{eq:GradEstimGreen}) are sharp.
Also, $\nabla_x G_U(x,y)$ is jointly continuous for $x\neq y\in U$, see \cite[(10)]{BKN}. 

Recall that $1<\alpha<2$.
We say that vector field $b:\Rd\to \Rd$
belongs to the Kato class $\pK_d^{\alpha-1}$ if 
\begin{equation}\label{eq:Kc}
\lim_{\eps \to 0} \sup_{x \in \RR^d} \int_{|x-z| < \eps}
|b(z)|\,|x-z|^{\alpha-1-d}\,dz = 0\,.
\end{equation}
For instance, if $b$ is bounded or if $|b(z)|\leq |z|^{1-\alpha+\varepsilon}$
and $0<\varepsilon<\alpha-1$, then $b \in \pK_d^{\alpha-1}$. 
Without much mention elements of $\pK_d^{\alpha-1}$ will either be vector fields $\Rd\to \Rd$ or real-valued test functions $\Rd\to \RR$, i.e. $\pK^{\alpha-1}_d$ is more a condition than a class.
Since $|x-z|^{\alpha-1-d}$ is locally bounded from below, $|b(z)|dz$ is a locally finite measure, and (\ref{eq:Kc}) is a local uniform integrability condition. 
If $b\in \pK^{\alpha-1}_d$ and $f$ is bounded, then $fb\in \pK^{\alpha-1}_d$, in particular, $fb$ is locally integrable.
We note that $ \pK_d^{\alpha-1}\subset\pK^\alpha_d$, where $\pK^\alpha_d$ is defined by
\begin{equation}\label{eq:akc}
\lim_{\eps \to 0} \sup_{x \in \RR^d} \int_{|x-z| < \eps}
|b(z)|\,|x-z|^{\alpha-d}\,dz = 0\,.
\end{equation}

Following \cite{MR2283957} and \cite{2009-TJ-KS-jee}
we recursively define,
for $t>0$ and $x,y \in \RR^d$, 
$$p_0(t,x,y)  =  p(t,x,y)\,,$$ 
$$
 p_n(t,x,y)  =  \int_0^t \int_{\RR^d} p_{n-1}(t-s,x,z) b(z) \cdot \nabla_z p(s,z,y)\,dz\,ds\,,\quad n \ge 1\,,
$$
and we let
\begin{equation}
\tp=\sum_{n=0}^\infty p_n\,.
\end{equation}
The series converges absolutely, $\tp$ is a continuous  probability transition density function,  and
\begin{equation}\label{ptxy_comp}
  c_T^{-1} p(t,x,y) \le \tp(t,x,y) \le c_Tp(t,x,y)\,,\qquad x,y \in \RR^d\,,\; 0<t<T\,,
\end{equation}
where $c_T \to 1$ if $T \to 0$, see \cite[Theorem 2]{MR2283957}. 
From a general perspective the approach of \cite{MR2283957}, \cite{2009-TJ-KS-jee} consist of using the semigroup as test functions, setting the assumptions on the perturbation so that 
$p_1$ is dominated by $p$ in short time, and recursively estimating multiple integrals defining $p_n$, so that the comparability with $p$ is preserved.
Auxiliary estimates of $\nabla_z p(s,z,y)$ are obtained in \cite{MR2283957, 2009-TJ-KS-jee} by subordination to the Gaussian kernel, but the scope of the method is wider. For instance applications to {S}chr\"odinger perturbations of general transition densities are given in \cite{MR2457489}, \cite{MR2507445}.

We let $\tPP$, $\tEE$ be the Markov distributions and expectations defined by transition density $\tp$ on the canonical path space.
We define the heat kernel of $L$ on $D$ by the usual G. Hunt's formula,
\begin{equation}\label{eq:Hunt}
\tp_D(t,x,y) = \tp(t,x,y) - \tEE^x\left[\tau_D < t;\; \tp(t-\tau_D,
  X_{\tau_D},y)\right]\,.
\end{equation}
We denote by $\tG_D(x,y)$ and $\tG_D$ the Green function and operator
of $L$ on $D$,
\begin{equation}\label{eq:deftG}
\tG_D(x,y)=\int_0^\infty \tp_D(t,x,y)dt\,,
\end{equation}
$$
\tG_D \phi(x)=\int_\Rd \tG_D(x,y)\phi(y)dy\,.
$$
By Blumenthal's 0-1 law, $\tp_D(t,x,y)=0$ and $\tG_D(x,y)=0$ if $x\in D^c$ or $y\in D^c$, see (\ref{ptxy_comp}).
The next lemmas rely on the definition of $\tp$ and generalize results stated above for $\Delta^{\alpha/2}$. The proofs of Lemma~\ref{l:wkp},~\ref{lem:lsgp} and~\ref{lem:tgi} are moved to the Appendix. 
\begin{lem}\label{l:wkp}
For $s>0$, $x\in D$ and $\phi \in C_c^\infty\big((0,\infty)\times D\big)$ we have 
\begin{equation}\label{tp_D_fs}
\int_s^\infty\!\!\!\! \int_{D} \tp_D(u-s,x,z) \left(\partial_u + \Delta_z^{\alpha/2} + b(z) \cdot \nabla_z\right) \phi\,(u,z) \,dz\,du = - \phi(s,x)\,.
\end{equation}
\end{lem}
By (\ref{ptxy_comp}) we have 
$$
\lim_{t \to 0} \frac{\tp(t,x,y)}{t} = \lim_{t \to 0}
\frac{p(t,x,y)}{t} = \nu(y-x)\,.
$$
Thus the intensity of jumps of the canonical process $X$ under $\tPP^x$ is the same as under $\PP^x$. Accordingly, we obtain the following description.
\begin{lem}\label{lem:lsgp}
The $\tPP^x$-distribution of $(\tau_D,X_{\tau_D})$ on $(0,\infty)\times (\overline{D})^c$ has density
\begin{equation}
  \label{eq:IWft}
  \int_D \tp_D(u,x,y)\nu(z-y)\,dy\,,\quad u>0\,,\;\delta_D(z)>0\,.
\end{equation}
\end{lem}
We define  the Poisson kernel of $D$ for $L$,
\begin{equation}\label{eq:djpt}
\tilde{P}_D(x,y)=\int_D \tG_D(x,z)\nu(y-z)\,dz\,,\quad x\in D\,,\;y\in D^c\,.
\end{equation}
By (\ref{eq:deftG}), (\ref{eq:djpt}) and (\ref{eq:IWft})  we have 
\begin{equation}\label{eq:IWt}
\tPP^x(X_{\tau_D} \in A)    =  \int_A\int_D \tG_D(x,z)\nu(y-z)\,dz\,dy =\int_A \tP_D(x,y)dy\,,
\end{equation}
if $A \subset (\bar{D})^c$. For the case of $A\subset \partial D$, we refer the reader to Lemma~\ref{l:nu}.

The following rough estimate of $\tG_D$ results from the estimates of $\tp$
and the fact that $X$ jumps out of $D$ at least with intensity 
$\int_{|y|>\diam(D)} \nu(y)dy>0$.
\begin{lem}\label{lem:GLDupper}
$\tG_D(x,y)$ is continuous for $x\neq y$, $\tG_D(x,x)=\infty$ for $x\in D$, and 
$$
\tG_D(x,y) \le C_0|x-y|^{\alpha-d}\,,\quad x,y\in \Rd\,,
$$
where $C_0=C_0(d,\alpha,{\rm diam}(D))$.
\end{lem}
\begin{proof}
We claim that there are constants $c$ and $C$ such that
\begin{equation}\label{eq:eotp}
\tp_D(t,x,y) \le Ce^{-ct}\,,\quad t>1, \quad x,y \in
  \RR^d\,.
\end{equation}
Indeed, let $\kappa_D(y)=\int_{D^c}\nu(z-y) dz$, so that $\kappa_D(y)\geq c>0$ for $y\in D$. Let $x\in D$, $t\geq 0$, and 
$F(t)=\tPP^x(\tau_D>t)=
\int_D \tp_D(t,x,y)dy$. 
By Lemma~\ref{lem:lsgp},
$$
-F'(t)= \int_{D} \tp_D(t,x,y)\kappa_D(y)dy\geq cF(t)\,,
$$
hence $\tPP^x(\tau_D > t) \le e^{-ct}$. 
By the semigroup property and (\ref{ptxy_comp}),  for $t>1$,
  \begin{align*}
    \tp_D(t,x,y) &\leq \int_D \tp_D(t-1,x,z)\tp(1,z,y) \,dz \\
    &\le c_1\int_D \tp_D(t-1,x,z)p(1,z,y) \,dz\\
    &\le c_1 p(1,0,0)\, \tPP^x(\tau_D > t-1) \le C e^{-c t}\,.
  \end{align*}
By (\ref{eq:deftG}),   (\ref{ptxy_comp}) and (\ref{eq:eotp}) we obtain
  \begin{align*}
    \tG_D(x,y)
    & \le \int_0^1 c_1 p(t,x,y)\,dt + \int_1^\infty Ce^{-ct}\, dt \\
    & \le \mathcal{A}_{d,\alpha} |x-y|^{\alpha-d} + C/c \le \left(\mathcal{A}_{d,\alpha} + C\diam(D)^{d-\alpha}/c\right)|x-y|^{\alpha-d}\,.
  \end{align*}
By (\ref{eq:deftG}) and dominated convergence theorem $\tG_D(x,y)$ is continuous if $x\neq y$, see (\ref{ptxy_comp}) and (\ref{eq:oppt}).
\end{proof}

The next lemma results from integrating (\ref{tp_D_fs}) against time.
\begin{lem}\label{lem:tgi}
For all $\varphi \in C_c^\infty(D)$ and $x\in D$ we have 
\begin{equation}\label{tG_D_fs}
\int_{D} \tG_D(x,z) \left(\Delta^{\alpha/2}\varphi(z) + b(z) \cdot \nabla \varphi(z) \right)\,dz = - \varphi(x)\,.
\end{equation}
\end{lem}

The definition of $L$-harmonicity is analogous to that of $\alpha$-harmonicity.
\begin{df}
$u$ is $L$-harmonic on $U$ if for each open bounded  
$V\subset\overline{V} \subset U$,
$$
u(x) = \tEE^xu(X_{\tau_V}), \quad x \in V.
$$
We say that $u$ is
regular $L$-harmonic on $U$ if also
$$
u(x) = \tEE^xu(X_{\tau_U})
\,, \quad x \in U\,.
$$
\end{df}
Here $\tEE^xu(X_{\tau_U})=\tEE^x[\tau_U<\infty;\, u(X_{\tau_U})]$ and we  always assume absolute integrability.
In particular, $x\mapsto \tG_D(x,y)$ is $L$-harmonic in $D\setminus\{y\}$, in fact
$\tG_D(x,y)=\tG_U(x,y)+\tEE^x \tG_V(X(\tau_U),y)$ for every open $U\subset D$.
We should note that in general $\tG_D(x,y)\neq \tG_D(y,x)$ (non-symmetry), and $y\mapsto \tG_D(x,y)$ is {\it not} $L$-harmonic. This accounts in part for the difficulties in estimating $\tG_D$.

\section{Perturbation formula}\label{s:3}
As before, $D$ is a bounded $C^{1,1}$ open set in $\Rd$, $b\in \pK_d^{\alpha-1}$ and $1<\alpha<2\leq d$. Let
$$\mbox{$\GD=G_D$, $\tGD=\tG_D$ and $\deltaD=\delta_D$.}$$
In view of (\ref{eq:GradEstimGreen}) the next lemma yields uniform integrability of $b(z)\cdot \nabla_z G(y,z)$.
In particular, the singularity $\delta(z)^{\alpha/2-1}$ of $\nabla_z \GD$  at $\partial D$ integrates against $|b|$.
 \begin{lem}\label{lem:UnifInt}
$\GD(y,z)/[\delta(z)\land |y-z|]$ is uniformly in $y$ 
  integrable against $|b(z)| dz$.
\end{lem}
\begin{proof}
In view of   (\ref{GreenEstimatesTJ}) it is enough to prove the uniform integrability of
$$
H(y,z)=
|y-z|^{\alpha-d}\frac{\delta(y)^{\alpha/2}\delta(z)^{\alpha/2}}
{[\delta(z)\vee |y-z|\vee\delta(y)]^\alpha}
\frac{1}{\delta(z)\land |y-z|}\,.
$$
Let $A_R(y) = \{z \in D \colon
H(y,z) > R\}$ for $R>0$. We will verify that
$$
\lim_{R \to \infty}\sup_{y \in D} \int_{A_R(y)} H(y,z)
|b(z)|\,dz =0\,.
$$ 
For $r>0$ we denote 
$$
K_r=\sup_{x\in \Rd} \int_{B(x,r)}|b(y)||x-y|^{\alpha-1-d}dy\,.
$$
By (\ref{eq:Kc}) we have that $K_r<\infty$ and $K_r \downarrow 0$ as $r \downarrow 0$. 
For all $x \in \RR^d$ and $r>0$,
$$
\int_{B(x,r)} |b(z)|\,dz \le r^{d+1-\alpha} \int_{B(x,r)}
|x-z|^{\alpha-1-d}|b(z)|\,dz \le K_r r^{d+1-\alpha}\,.
$$
For $r>0$ we let
$D(r)=\{z\in D:\,\delta(z)>r\}$. If $m>0$, $y\in D$ and $z\in D(\delta(y)/m)$, that is $\delta(z)>\delta(y)/m$, then we have
$$|H(y,z)| \le m^{1-\alpha/2}
|y-z|^{\alpha-1-d}\,.
$$ 
Indeed, by (\ref{eq:max}),
\begin{eqnarray*}
H(y,z)&=&
|y-z|^{\alpha-d}\frac{\delta(y)^{\alpha/2}\delta(z)^{\alpha/2}}
{[\delta(z)\vee |y-z|\vee\delta(y)]^\alpha}
\frac{\delta(z)\vee |y-z|}{\delta(z)|y-z|}\\
&\leq&
|y-z|^{\alpha-1-d}\delta(z)^{\alpha/2-1}\delta(y)^{\alpha/2}
[\delta(z)\vee |y-z|\vee \delta(y)]^{1-\alpha}
\\
&\leq&
|y-z|^{\alpha-1-d}\delta(z)^{\alpha/2-1}\delta(y)^{1-\alpha/2}
\leq
m^{1-\alpha/2}|y-z|^{\alpha-1-d}\,.
\end{eqnarray*}
If $R \to \infty$, then uniformly in $y$ we have
\begin{align}
  &\int_{D({\delta(y)}/m) \cap A_R(y)}  H(y,z) |b(z)|\,dz  \nonumber\\
  &\le m^{1-\alpha/2}\int_{\{z \in D \colon |y-z|^{\alpha-1-d} >
    Rm^{\alpha/2-1}\}} |y-z|^{\alpha-1-d} |b(z)|\,dz \to
  0\,. \label{eq:unif1}
\end{align}
For $y\in D$,  $k,n \ge 0$ and $m \ge 2$ we consider
$$
W^m_{n,k}(y) = \{z \in D\colon
\frac{\delta(y)}{m2^{n+1}} < \delta(z) \le \frac{\delta(y)}{m2^{n}}
,\; k \delta(y) < |y-z|\le (k+1) \delta(y) \}.
$$
$W^m_{n,k}(y)$ may be covered by $c_1 (k + 1)^{d-2} m^{d-1}2^{n(d-1)}$ balls of
radii $\frac{\delta(y)}{m2^n}$, thus
\begin{align*}
  \int_{W^m_{n,k}(y)}|b(z)|\,dz & \le c_1 (k + 1)^{d-2} m^{d-1}2^{n(d-1)} \sup_{x \in \RR^d} \int_{B(x,\delta(y)/m2^n)} |b(z)|\,dz \\
  & \le c_1 K_{\delta(y)/m2^n} (k + 1)^{d-2} m^{d-1}2^{n(d-1)} \left(\frac{\delta(y)}{m2^n}\right)^{d+1-\alpha}\\
  & = c_1 K_{\delta(y)/m2^n} (k + 1)^{d-2} m^{\alpha-2}
  2^{n(\alpha-2)}\delta(y)^{d+1-\alpha}\,.
\end{align*}
For $z\in W^m_{n,k}(y)$ we have $\delta(y)\geq 2\delta(z)$, hence $|y-z|\ge \delta(y)/2$ and $|y-z|\ge\delta(z)$. We obtain
\begin{align*}
  &\int\limits_{A_R(y) \setminus D({\delta(y)}/m)} H(y,z) |b(z)|\,dz 
\le \sum_{n=0}^\infty\sum_{k=0}^{\infty}\int\limits_{W^m_{n,k}(y)} \frac{\delta(y)^{\alpha/2}}{|y-z|^d \delta(z)^{1-\alpha/2}}|b(z)|\,dz\\
  &\le  \sum_{n=0}^\infty\sum_{k=0}^{\infty}\int_{W^m_{n,k}(y)} \frac{\delta(y)^{\alpha/2}}{\big((k+1)\delta(y)/2\big)^{d} [\delta(y)/(m2^{n+1})]^{1-\alpha/2}}|b(z)|\,dz\\
  & \le c_2K_{\delta(y)/m}\sum_{n=0}^\infty\sum_{k=0}^{\infty}(k + 1)^{d-2} m^{\alpha-2}2^{n(\alpha-2)}
  2^d(k+1)^{-d}2^{n(1-\alpha/2)}m^{1-\alpha/2}\\
& \le
  c_3m^{\alpha/2-1}K_{\delta(y)/m}\,.
\end{align*}
Let $\varepsilon >0$. We chose $m$  and $R$ so large that $c_3m^{\alpha/2-1}K_{\diam(D)/m} <
\varepsilon/2 $ and
$$
\sup_{y \in D}\int_{D({\delta(y)/m}) \cap A_R(y)} H(y,z)
|b(z)|\,dz <\varepsilon/2\,.
$$
This completes the proof.
\end{proof}

We consider the operator $\Ab$:
$$
(\Ab \phi)(x) = b(x)\cdot \nabla\phi(x) 
= \sum_{i=1}^d b_i(x) \frac{\partial \phi(x)}{\partial x_i}.
$$ 
We will  study the perturbation series
$\sum_{n=0}^\infty
(G \Ab )^n G$ of integral operators on $\mathcal{K}_d^{\alpha-1}$.
Namely we will apply $G\Ab G$ to real-valued $f\in \pK_d^{\alpha-1}$:
\begin{eqnarray}
  G\Ab Gf(x) & = & 
  \int_D G(x,z) b(z) \cdot \nabla_z \int_D G(z,y)f(y)\,dy\,dz\,. \label{operatorGD_bG}
\end{eqnarray}
We will need to interchange the integration and  differentiation in (\ref{operatorGD_bG}).
\begin{lem}\label{GradGreen}           
  Let $1<\alpha<2$. If $f \in \mathcal{K}_d^{\alpha-1}$ or at least $G(z,y)/|y-z|$ is locally in $z\in D$ uniformly integrable against $|f(y)|dy$, then
  \begin{equation}
    \nabla_z\int_D\GDD(z,y)f(y)\,dy = \int_D \nabla_z\,\GDD(z,y)f(y)\,dy\,,\quad z \in D\,.
  \end{equation}
\end{lem}
\begin{proof}
 The result is proved for $f\in \pK^{\alpha-1}_d$ in \cite[Lemma 5.2]{BKN}. For the more general $f$ we note that $G(z,y)$ and $\nabla_z G(z,y)$  are continuous on $D\times D$ except at $z=y$, see the remark following (\ref{eq:GradEstimGreen}). They are also uniformly integrable against $|f(y)|dy$ for $z$ in compact subsets of $D$. In consequence $g(z)=Gf(z)$ and $k(z)= \int_D \nabla_z\,\GDD(z,y)f(y)\,dy$ are continuous on $D$.
We consider $f_n=f \land n \lor (-n)\in \pK^{\alpha-1}_d$, $g_n=Gf_n$.
We have $g_n\to g$ and $\nabla g_n\to k$. It follows that $\nabla g=k$.
\end{proof}

For $x,y\in D$ we let
\begin{align}
   & {\mk}(x,y)=\int_D  |b(z)|\frac{G(x,z)G(z,y)}{G(x,y)(\deltaDD(z) \land |y-z|)} \,dz\,, \label{3GIntegral} \\
   & {\mgk}(x,y)=\int_D  |b(z)|\frac{G(x,z)G(z,y)(\deltaDD(x) \land |x-y|)}{G(x,y)(\deltaDD(z) \land |y-z|)(\deltaDD(x) \land |x-z|)} \,dz\,. \label{nabla3GIntegral}
\end{align}
In what follows $\mk$ and $\mgk$  will serve as majorants for the perturbation series.
\begin{lem}\label{Lem3GIntegral}
Let $\lambda, r<\infty$. There is
  $\CI=\CI(d,\alpha,b,\lambda,r)$ such that if $D$ is $C^{1,1}$,
${\rm diam}(D)/r_0(D)\le\lambda$ and $\diam(D)\le r$, then 
${\mk}(x,y) \le \CI$, ${\mgk}(x,y) \le 2\CI$ for $x,y \in D$, 
and $\CI(d,\alpha,b,\lambda,r)\to 0$ as $r\to 0$.
\end{lem}
\begin{proof} 
Denote $g(x) = \delta(x)^{\alpha/2}$. 
Let $x,z\in D$.
By  (\ref{GreenEstimatesTJ}) we have
\begin{align}
& \frac{g(z)}{g(x)}G(x,z)\approx \frac{g^2(z)}{(\delta(x)\lor |x-z|\lor \delta(z))^\alpha} 
|x-z|^{\alpha-d} \nonumber \\
& \le \left(\frac{\delta(z)}{\delta(z)\lor |x-z|}\right)^\alpha |x-z|^{\alpha-d} 
\le \frac{\delta(z)}{\delta(z)\lor |x-z|}|x-z|^{\alpha-d} \label{eq:alphag1}\\
&= \big( \delta(z) \land |x-z|\big)|x-z|^{\alpha-1-d}\,. \label{eq:ggG}
\end{align}
Let $\mathcal{G}(x,y)=G(x,y)/[g(x)g(y)]$. The so-called 3G Theorem 
holds for $\mathcal{G}$: 
$$
\mathcal{G}(x,z)\land \mathcal{G}(z,y)\leq c \mathcal{G}(x, y)\,,\quad x,y,z\in D\,,
$$
where $c $ depends only on $d$, $\alpha$ and the distortion of $D$ (\cite{MR2160104}).
We obtain
\begin{align}
  & \frac{G(x,z)G(z,y)}{G(x,y)} =g^2(z) \frac{\mathcal{G}(x,z)\mathcal{G}(z,y)}{\mathcal{G}(x, y)}
  \le c \left(\frac{g(z)}{g(x)}G(x,z) \;\lor \; \frac{g(z)}{g(y)}G(z,y)\right)  \nonumber\\
  &\le  c  \left(\frac{\delta(z) \land |x-z|}{|x-z|^{d+1-\alpha}} \lor  \frac{\delta(z) \land |y-z|}{|y-z|^{d+1-\alpha}} \right) \nonumber\\
  & =  c  \big( \delta(z) \land |x-z| \land |y-z| \big)
  \left(\frac{1}{|x-z|^{d+1-\alpha}} \lor  \frac{1}{|y-z|^{d+1-\alpha}}
  \right)\,. \label{eq:3G}
\end{align}
We have
\begin{equation*}
\frac{G(x,z)G(z,y)}{G(x,y)\big( \delta(z) \land |y-z| \land |x-z|\big)} \le
c  (|x-z|^{\alpha-1-d} + |y-z|^{\alpha-1-d})\,,
\end{equation*}
so we actually have uniform integrability against $|b(z)|dz$. The statement about $\kappa$ follows form (\ref{eq:Kc}).

To estimate ${\mgk}$ we consider two cases. If $2|x-z| > \delta(x) \land |x-y|$, then
$$ 
\frac{\delta(z) \land |x-z| \land |y-z|}{(\delta(z) \land |y-z|)(\delta(x)  \land |x-z|)} \le \frac{1}{\delta(x)  \land |x-z|} \le \frac{2}{\delta(x)  \land |x-y|}\,.
$$
If $2|x-z| \le \delta(x) \land |x-y|$, then $\delta(z) \ge \delta(x)/2$, $|y-z| \ge |x-y|/2$, and so
$$
\frac{\delta(z) \land|x-z|  \land |y-z|}{(\delta(z) \land |y-z|)(\delta(x)  \land |x-z|)} \le \frac{\delta(z) \land |x-z|  \land |y-z|}{(\delta(x)/2 \land |x-y|/2)|x-z|} \le \frac{2}{\delta(x)  \land |x-y|}\,.
$$
By (\ref{eq:3G}) and (\ref{eq:Kc}) we obtain ${\mgk}(x,y) \le 2\CI$. In fact we observe the uniform integrability against $|b(z)|dz$.
The above estimates of the factors in (\ref{3GIntegral})  and (\ref{nabla3GIntegral}) depend on $D$ only through $d$ and ${\rm diam}(D)/r_0(D)$ (\cite{MR1991120}).
Therefore the integrals in  (\ref{3GIntegral})  and (\ref{nabla3GIntegral}) are arbitrarily small if  ${\rm diam}(D)$ is small enough, and the distortion of $D$ is bounded by a constant. This follows from (\ref{eq:Kc}). If ${\rm diam}(D)$ is not small but finite then we only have the boundedness of ${\mk}$ and $\mgk$, which also follows from (\ref{eq:Kc}). 
\end{proof}
For $x\neq y$ we let
\begin{equation}\label{eq:defG1}
\kk_1(x,y) = \int_D G(x,z)
b(z) \cdot \nabla_z G(z,y) \,dz.
\end{equation}
By (\ref{eq:GradEstimGreen}), (\ref{3GIntegral}) and Lemma~\ref{Lem3GIntegral} the integral is absolutely convergent,
\begin{equation}\label{kxyEstimates}
  |\kk_1(x,y)| 
  \le d\, G(x,y) \int_D    \frac{|b(z)|G(x,z)G(z,y)}{G(x,y)(\deltaDD(z) \land |y-z|)} \,dz
\le dC_1 G(x,y)
\,. 
\end{equation}
For $f\in \pK_d^{\alpha-1}\subset \pK^\alpha_d$ we  have
\begin{eqnarray*}
&&\int_D  G(x,y) \int_D  |b(z)| \frac{G(x,z)G(z,y)}{G(x,y)(\deltaDD(z) \land |y-z|)} \,dz |f(y)|dy\\
  & \leq & C_1 \int_D G(x,y) |f(y)|dy<\infty \,,
  \end{eqnarray*}
hence by  Lemma~\ref{GradGreen}, (\ref{eq:GradEstimGreen}) and Fubini's theorem,
\begin{eqnarray*}
  G\Ab Gf(x) & = & \int_D G(x,z)  \int_D b(z) \cdot \nabla_z G(z,y)f(y)\,dy\,dz \\
  & = & \int_D \kk_1(x,y) f(y)\,dy\,.
\end{eqnarray*}
We like to note that the linear map $f\mapsto b\nabla Gf$ preserves $\pK^{\alpha-1}_d$ because
$\nabla Gf$ is a bounded function, see Lemma~\ref{lem:UnifInt} and the remarks following (\ref{eq:Kc}).

We will now prove the pointwise perturbation formula.
\begin{lem}\label{lem:pf}
Let $x,y\in \Rd$, $x\neq y$. We have
\begin{equation}\label{eq:wp}
\tGDD(x,y)=\GDD(x,y)+\int_D \tGDD(x,z)b(z)\cdot \nabla_z\GDD(z,y)dz\,.
\end{equation}
\end{lem}
\begin{proof}
Let $x\in D$. For $\phi\in \pK^{\alpha-1}_d$ we consider
$$
\Lambda(\phi)= 
\int_\Rd \left [\tGDD(x,y)-\GD(x,y)-\int_\Rd \tGDD(x,z)b(z)\cdot \nabla_z\GDD(z,y)dz\right]\phi(y)dy\,.
$$
By Lemma~\ref{lem:UnifInt} and Lemma~\ref{lem:GLDupper}, the {\it  iterated}\/ integral converges absolutely.
If $\varphi\in C^\infty_c(D)$ and $\phi=\Delta^{\alpha/2}\varphi$, then using (\ref{tG_D_fs}), (\ref{eq:fg}) and Lemma~\ref{GradGreen} we obtain
\begin{eqnarray*}
\Lambda(\phi)=-\varphi(x)-\tGDD(b\nabla\varphi)(x) +\varphi(x) -\tGDD (b \nabla(-\varphi))(x)=0\,.
\end{eqnarray*}
By \cite[Theorem 3.12]{MR1671973}, $\Lambda(\phi)=\int_\Rd \phi(y)\lambda(y)dy$, where 
$\lambda$ is $\alpha$-harmonic on $D$. Since our $\lambda$ is bounded near $\partial D$ and vanishes on $D^c$, we have that $\lambda\equiv 0$, see \cite[Lemma 17]{MR1671973}. 
By uniform integrability and the remark following (\ref{eq:GradEstimGreen}) and by Lemma~\ref{lem:GLDupper} we see that both sides of (\ref{eq:wp}) are continuous in $y\in D\setminus \{x\}$, hence we have pointwise equality in (\ref{eq:wp}).
\end{proof}
In addition to $\kk_1$ we inductively define
\begin{equation*}
\kk_{n}(x,y)=\int \kk_{n-1}(x,z) b(z)\cdot \nabla_z G(z,y)\,dz\,,\quad x\neq y\in D\,,\quad n=2,3,\ldots\,.
\end{equation*}
We also let $\kk_0(x,y)=G(x,y)$.
By (\ref{eq:GradEstimGreen}),  (\ref{kxyEstimates}), (\ref{3GIntegral}), Lemma~\ref{Lem3GIntegral} and {induction},
\begin{eqnarray}
  &&|\kk_{n}(x,y)|  \le  \int_D  |\kk_{n-1}(x,z)| |b(z)| |\nabla_z G(z,y)| \,dz \nonumber  \\ 
  && \le  (\CI d)^{n-1} \int_D  |b(z)|  G(x,z)  |\nabla_z G(z,y)| \,dz  \le (\CI d)^{n} G(x,y)\,,\label{eq:eGn}
\end{eqnarray}
where $n=0,1, 2, \ldots$ and $x\neq y$.
By (\ref{eq:eGn}) and induction we prove that
$$
\int_D \kk_n (x,z) \int_D b(z)\cdot \nabla_z G(z,y) f(y)\,dydz=
\int_D \kk_{n+1}(x,y) f(y)\,dy\,,
$$
hence $G_n$ is the integral kernel of $G(\Ab G)^n$, 
\begin{equation*}
  G(\Ab G)^n f(x) = \int_D \kk_n(x,y) f(y)\,dy\,, \quad f \in \pK_d^{\alpha-1}\,,\quad x\in D\,.
\end{equation*}
We can also handle the gradient of $G_n$. 
Namely, for $x,y \in D$, $x\neq y$, and $n=1,2,\ldots$, we have:
\begin{equation}\label{eq:gn1}
|\nabla_x G_{n-1}(x,y)|\leq (2\CI)^{n-1} d^n\frac{\GDD(x,y)}{\delta(x) \land |x-y|}\,,
\end{equation}
\begin{equation}\label{eq:gn2}
G_n(x,y)=\int G(x,z)b(z)\cdot \nabla_z G_{n-1}(z,y)dz\,,
\end{equation}
and
\begin{equation}\label{eq:gn3}
\nabla_x G_n(x,y)= \int \nabla_x G_{n-1}(x,z)b(z)\cdot \nabla_z G(z,y)dz\,.
\end{equation}
The inequality and the equalities are proved consecutively by induction.
In the process we use Lemma~\ref{GradGreen},  (\ref{nabla3GIntegral}),  estimates following  (\ref{eq:3G}) in the proof of Lemma~\ref{Lem3GIntegral}, and Fubini's theorem.

Iterating (\ref{eq:wp}), by (\ref{eq:gn2}) we obtain for $n=0, 1,\ldots$, and $x\neq y$,
\begin{eqnarray}
\tG(x,y)&=&G(x,y)+\int \tG(x,z)b(z)\cdot \nabla_zG(z,y)dz\nonumber\\
&=&\sum_{k=0}^n \kk_k(x,y)
+\int \tG(x,z)b(z)\cdot \nabla_{z}G_n(z,y)dz\,.\label{eq:ipf}
\end{eqnarray}
The details are left to the reader.

\section[Local results] {Local results}
\label{chap:Green}

We will  prove and use the comparability of $G_S$ and $\tG_S$ for small smooth sets, $S$, of class $C^{1,1}$ (this part of our development is similar to \cite{MR2140204}).
The following is a variant of Khasminski's lemma (\cite{MR1329992}).
\begin{lem}\label{Theorem1s}
Let $d\geq 2$, $1<\alpha<2$, $b\in \pK_d^{\alpha-1}$ and $\lambda>0$. There is $\varepsilon=\varepsilon(d,\alpha,b,\lambda)>0$ such that if 
${\rm diam}(S)/r_0(S)\le \lambda$ and $\diam(S)\le\varepsilon$, then
  \begin{equation}
    \label{eq:egfs}
\frac{2}{3}G_S(x,y) \le \tG_S(x,y) \le \frac{4}{3} G_S(x,y), \qquad x,y \in
\Rd\,.
  \end{equation}
\end{lem}
\begin{proof}
Let $D=S$.
By Lemma~\ref{Lem3GIntegral} and (\ref{eq:eGn}) there is $\varepsilon=\varepsilon(d,\alpha,b,\lambda)>0$ and
$$
|G_n(x,y)|\leq 4^{-n} G(x,y)\,,\quad x\neq y\,,\; n=0,1,\ldots,
$$
provided $\diam(D)/r_0(D)\le \lambda$ and $r_0(D)\le \varepsilon$.
For $x\neq y$ we have $\tG(x,y)=\sum_{n=0}^\infty \kk_n(x,y)$. 
Indeed, the remainder in (\ref{eq:ipf}) is bounded by
\begin{eqnarray*}
c\int_D |x-z|^{\alpha-d}|b(z)|  (2C_1)^{n-1}d^n\frac{G(z,y)}{\delta(z)\land |y-z|}dz\to 0\,,\quad \mbox{ as } n\to \infty\,.
\end{eqnarray*}
Here $2C_1 d\leq 1/2$ and the integral is finite because of Lemma~\ref{lem:UnifInt} and (\ref{eq:akc}). Thus,
$$
\tG(x,y)
=\sum_{n=0}^\infty \kk_n(x,y)\le
\sum_{n=0}^\infty 4^{-n}G(x,y) = \frac{4}{3}G(x,y)\,,
$$
and
$$
\tG(x,y)\geq
G(x,y) - \sum_{n=1}^\infty 4^{-n}G(x,y)=\frac{2}{3}G(x,y)\,.
$$

\end{proof}
We like to note that the comparison constants in the above proof will improve to $1$ if  ${\rm diam}(S)\to 0$ and the distortion of $S$ is bounded.
By (\ref{eq:IWt}),
\begin{equation}
  \tPP^x(X_{\tau_S} \in A) \approx \PP^x(X_{\tau_S} \in A)\,,\quad x\in S\,,\quad A\subset (\overline{S})^c\,. \label{I-WComp1}
\end{equation}

We are in a position to prove that the boundary of our general $C^{1,1}$ open set $D$ is not hit at the first exit (recall that $1<\alpha<2$).
\begin{lem}\label{l:nu}
For every $x\in D$ we have that $\tPP^x(X_{\tau_D}\in \partial D)=0$.
\end{lem}
\begin{proof}
Let $u(x)=\tPP^x(X_{\tau_D}\in \partial D)$, $x\in \Rd$.
We claim that there exists $c=c(d,\alpha,D, b)>0$ such that $u(x)<1-c$ for $x\in D$.
Indeed, we consider small $\varepsilon>0$, $x\in D$,  $r=\varepsilon\, {\rm dist}(x,D^c)$, the ball $B=B(x,r/2)\subset D$, and a ball $B'\subset (\overline{D})^c$ with radius and distance to $B$ comparable with $r$. By (\ref{I-WComp1}) and (\ref{eq:poisson:ball}),
$$\tPP^x(X_{\tau_D}\notin \partial D)
\geq \tPP^x(X_{\tau_B(x,r/2)}\in B')\approx \PP^x(X_{\tau_B(x,r/2)}\in B')
\geq c\,.$$ 
 Furthermore, let $D_n=\{y\in D:\,{\rm dist}(y,D^c)>1/n\}$, $n=1,2,\ldots$.
We consider $n$ such that $B(x,r/2)\subset D_n$.
We have 
$\tPP^x(X_{\tau_{D_n}}\in \overline{D})\leq 1- \tPP^x(X_{\tau_B}\in B')\leq 1-c$, as before.
Let $C=\sup \{u(y): y\in D\}$. We have
$u(x)=\tEE^x\{u(X_{\tau_{D_n}});\,X_{\tau_{D_n}}\in \overline{D}\}\leq C(1-c)$, hence
$C\leq C(1-c)$ and so $C=0$. 
\end{proof}
In the context of Lemma~\ref{Theorem1s}, 
the $\tPP^x$ distribution of $X_{\tau_S}$ is absolutely continuous with respect to the Lebesgue measure,
and has density function 
\begin{equation}\label{PoissonComp}
  \tilde P_S(x,y) \approx P_S(x,y)\,,\quad \;y\in S^c\,,
\end{equation}
provided $x\in S$.
This follows from (\ref{eq:IWt}) and Lemma~\ref{l:nu}.
For clarity, 
\begin{equation}
  \tPP^x(X_{\tau_S} \in A) \approx \PP^x(X_{\tau_S} \in A)\,,\quad x\in S\,,\quad A\subset S^c\,. \label{I-WComp}
\end{equation}

\begin{lem}[Harnack inequality for $L$]\label{HIforL}
  Let $x,y \in \Rd$, $0<s<1$ and $k\in \mathbb{N}$ satisfy $|x-y|
  \le 2^ks$. Let ${\tu}$ be nonnegative in $\Rd$ and
  $L$-harmonic in $B(x,s) \cup B(y,s)$.  There is $C = C(d, \alpha, b)$ such that
  \begin{equation}\label{LHarnackIneq}
    C^{-1}2^{-k(d+\alpha)}{\tu}(x) \le {\tu}(y) \le C 2^{k(d+\alpha)}{\tu}(x)\,.
  \end{equation}
\end{lem}
\begin{proof}
We may assume that $s\leq 1\land \varepsilon/2$, with  $\varepsilon$ of Lemma~\ref{Theorem1s}.
Let $u(z)=\tu(z)$ for $z\in B(y,2s/3)^c$ and $u(z)=\int_{B(y,2s/3)^c} {\tu}(v)P_{B(y,2s/3)}(z,v)\,dv$ for $z\in B(y,2s/3)$, so that
$u$ is nonnegative in $\RR^d$ and (regular) \ah{} in $B(y,2s/3)$. Let $z \in  B(y,s/2)$. 
By (\ref{I-WComp}),
 $$
    {\tu}(z) = \tEE^z {\tu}(X (\tau_{B(y,2s/3)})) = \int_{B(y,2s/3)^c} {\tu}(v)\tilde P_{B(y,2s/3)}(z,v)\,dv\approx u(z)\,.
 $$
By Lemma~\ref{HI} we get $\tilde{u}(z) \approx
  \tilde{u}(y)$ with a constant depending only on $d$,
  $\alpha$ and $b$. 
To compare $\tu(x)$ and $\tu(y)$ we will assume that $|x-y| \ge 3s/2$, because otherwise we may take
  smaller $s$. For $z \in B(y,s/2)$ we have $|x-z| \le   |x-y|+|y-z| 
 \le2^k s + s/2  \le 2^{k+1}s$, and we get
  \begin{eqnarray*}
    P_{B(x,s/2)}(x,z) &=& C_{d,\alpha} \left(\frac{(s/2)^2}{|x-z|^2 - (s/2)^2}\right)^{\alpha/2} \frac{1}{|x-z|^d}\\
    & \ge & C_{d,\alpha} 2^{-\alpha} s^\alpha |x-z|^{-d-\alpha}\\
    & \ge & C_{d,\alpha} 2^{-(d+2\alpha)}2^{-k(d+\alpha)} s^{-d}\,. 
  \end{eqnarray*}
Since
   $\tilde P_{B(x,s/2)} \approx P_{B(x,s/2)}$, by the first part of the proof we obtain
  \begin{eqnarray*}
    {\tu}(x) & = & \int_{B(x,s/2)^c} \tilde P_{B(x,s/2)}(x,z){\tu}(z)\,dz 
\ge \int_{B(y,s/2)} \tP_{B(x,s/2)}(x,z) {\tu}(z) \,dz \\
    &\approx& \int_{B(y,s/2)} P_{B(x,s/2)}(x,z) {\tu}(z) \,dz \\
    & \ge & |B(y,s/2)|C_{d,\alpha}2^{-(d+2\alpha)}2^{-k(d+\alpha)} s^{-d} {\tu}(y)
    =  c2^{-k(d+\alpha)}{\tu}(y)\,.
  \end{eqnarray*}
  By symmetry, $\tu (x)\approx \tu(y)$.
\end{proof}
We obtain a boundary Harnack principle for $L$ and general $C^{1,1}$ sets $D$.

\begin{lem}[BHP]\label{BHPforL}
  Let $z \in \partial{D}$, $0<r\le  r_0(D)$, and $0<p<1$. If
$\tilde{u}, \tilde{v}$ are nonnegative in $\Rd$, regular $L$-harmonic
  in $D \cap B(z,r)$, vanish on $D^c \cap B(z,r)$
  and satisfy $\tilde{u}(x_0)=\tilde{v}(x_0)$ for some $x_0 \in D \cap B(z,pr)$
  then
  \begin{equation}
    \label{BHPforLEq}
    \CIII^{-1}\tilde{v}(x) \le \tilde{u}(x) \le \CIII \tilde{v}(x)\,,\quad x \in D \cap B(z,pr)\,,
  \end{equation}
with $\CIII = \CIII(d,\alpha,b,p,r_0(D))$.
\end{lem}
\begin{proof}  In view of  Lemma~\ref{HIforL} we may assume that $r$ is small.
Let $F=F(z,r/2) \subset B(z,r)$ be the $C^{1,1}$ domain of Lemma~\ref{l:loc}, localizing $D$ at $z$.
For $x
  \in F$ we have $\tilde{u}(x) = \int \tilde P_{F}(x,z) \tu(z)\,dz
  \approx u(x)$, where $u(x)= \int P_{F}(x,z) \tu(z)\,dz$.
Similarly $\tilde{v}(x) \approx v(x)=\int P_{F}(x,z) \tilde{v}(z)\,dz$. Since $\tilde{u}(x_0)
  = \tilde{v}(x_0)$, we have $u(x_0)\approx v(x_0)$.
    By Lemma~\ref{BHP}, ${u}(x) \approx {v}(x)$, provided $x\in D\cap B(z,r/8)$.
We use Lemma~\ref{HIforL} for the full range $x\in D\cap B(z,pr)$.
\end{proof}

\section{Proof of Theorem~\ref{Theorem1}}\label{sec:b}
 By (\ref{eq:wp}) and (\ref{eq:GradEstimGreen}) we have the estimate
\begin{equation}\label{eq:0}
  \tG(x,y) \le G(x,y) + d \int_{D} \frac{\tG(x,z)G(z,y)}{\deltaDD(z)\land |y-z|} |b(z)|\,dz\,,\quad
x,y\in D\,.
\end{equation}
We consider  $\eta<1$, say $\eta=1/2$.
By Lemma~\ref{lem:UnifInt} and the uniform integrability in Lemma~\ref{Lem3GIntegral} (see (\ref{eq:3G})) there is a constant $r>0$ so small that
\begin{equation}\label{eq:1}
\int_{D^r} \frac{G(z,y)}{\deltaDD(z)\land |y-z|} |b(z)|\,dz  <\frac{\eta}{d}\,,\quad
  y \in D\,,
\end{equation}
and
\begin{equation}\label{eq:2}
  \int_{D^r} \frac{G(x,z)G(z,y)}{G(x,y)(\deltaDD(z)\land |y-z|)} |b(z)|\,dz <\frac{\eta}{d}\,, \qquad y \in D\,.
\end{equation}
Here $D^r = \{z \in D \colon \deltaDD(z) \leq r\}$. 
We denote $$\rho=[\varepsilon \land r_0(D)\land r]/16\,,$$ with $\varepsilon=\varepsilon(d,\alpha,b,2/\kappa)$ of Lemma~\ref{Theorem1s}, see also Lemma~\ref{l:loc}. 

To prove (\ref{eq:egf}) we will consider $x$, $y$ in a partition of $D\times D$.  We will also consider $Q,R \in \partial D$ 
such that $\deltaDD(x) = |x-Q|$, $\delta(y)=|y-R|$.

I\@. First we suppose that $\deltaDD(y)\geq \rho /4$. We denote
  \begin{itemize}
\item $D_1= \{x\in D\colon\; |y-x|\leq \rho /8\}$,
\item $D_2 = \{x\in D\colon\;\delta(x)\ge \rho/8\}$,
\item $D_3 = \{x\in D \colon \; \deltaDD(x) <\rho /8\}$.

  \end{itemize}
\noindent a) Let $x \in D_1$.  Denote $B=B(y,\rho /4)$. By (\ref{GreenEstimates}) we have $G_B(x,y)  \approx |x-y|^{\alpha-d} \approx G(x,y)$. Lemma~\ref{Theorem1s} yields the lower bound in (\ref{eq:egf}):
$$\tGDD(x,y) \ge \tG_B(x,y) \approx G_B(x,y) \approx G(x,y)\,.$$
By Lemma~\ref{lem:GLDupper} we get the upper bound in (\ref{eq:egf}).\\

\noindent b) Let $x \in D_2\setminus D_1$. Let $x_0 \in D$ be such that $|x_0-y|=\rho /8$. $\tGDD(\cdot,y)$ is $L$-harmonic in $B(x,\rho /8) \cup B(x_0,\rho /8)$. 
By Lemma~\ref{HIforL}, a) and Lemma~\ref{HI} we get 
$\tGDD(x,y) \approx \tG(x_0,y)\approx G(x_0,y)  \approx G(x,y)$. \\

\noindent c) Let $x \in D_3$. Let $x_0,y_0 \in D$ be collinear with $x,Q$ and such that $\delta(x_0)=|x_0-Q|=5\rho/32$ and $\delta(y_0)=|y_0-Q|=7\rho/32$ (consider an inner ball tangent at $Q$ to see the situation). Let $F=F(Q,\rho)$ be the approximating domain of Lemma~\ref{l:loc}. The functions $\tG(\cdot,y_0)$, $\tG_F(\cdot,y_0)$ and $\tGDD(\cdot,y)$ are regular
$L$-harmonic in $B(Q,6\rho /32)$. 
By Lemma~\ref{BHPforL}, Lemma~\ref{Theorem1s} and Lemma~\ref{BHP},
$$
\frac{\tGDD(x,y)}{\tGDD(x_0,y)} \approx
\frac{\tG_F(x,y_0)}{\tG_F(x_0,y_0)} \approx \frac{G_F(x,y_0)}{G_F(x_0,y_0)}
\approx \frac{\GDD(x,y)}{\GDD(x_0,y)}\,.
$$
By a) and b), $\tGDD(x_0,y) \approx G(x_0,y)$ and we obtain (\ref{eq:egf}) in the considered case I, that is for $\delta(y)\ge \rho/4$ and all $x\in D$. 

Before we proceed to the next case we recall that $\tGDD$ is non-symmetric.\\

II\@. Suppose that $\deltaDD(y) \le \rho /4$.
The proof of (\ref{eq:egf}) follows in 3 steps. \\
Step 1. We will first prove that 
$\tGDD(x,y) \ge cG(x,y)$, $x\in D$.

To this end we denote
\begin{itemize}
\item $F_1 = \{x\in D\colon\; |x-R|\leq \rho\}$,
\item $F_2 = \{x\in D \colon \deltaDD(x) \ge \rho /4 \}$,
\item $F_3 = \{x\in D\colon\; \delta(x)<\rho/4\}$.
\end{itemize}

\noindent d) Let $x \in F_1$. Consider $F=F(R,8\rho)$. Then
$\deltaDD(x)=\delta_F(x)$ and $\deltaDD(y) =
\delta_F(y)$. Consequently, $G(x,y) \approx G_F(x,y)$, see (\ref{GreenEstimates}). 
As before we
have
$$
\tGDD(x,y) \ge \tG_F(x,y) \approx G_F(x,y) \approx G(x,y)\,.
$$

\noindent e) Let $x \in F_2\setminus F_1$. Let $x_0 \in D$ be collinear with $y$, $R$ and such that
$\deltaDD(x_0)=|x_0-R|=\rho /2$. We note that $|x-y| \ge 3\rho/4$ and $|x_0-y|\ge \rho/4$.
 By Harnack inequalities and d),
$\tGDD(x,y) \approx \tGDD(x_0,y) \ge cG(x_0,y) \approx G(x,y)$.\\

\noindent f) Let $x \in F_3\setminus F_1$. Let $z_0 \in D$ be such that
$\deltaDD(z_0) =|z_0-R|=\rho/3 $ and let $x_0 \in D$ be such that $\deltaDD(x_0)=|x_0-Q|=
\rho /4$. We have $|x_0-y|>\rho/2$ and $|z_0-y|\ge \rho/12$.
By Harnack inequalities and d),
\begin{equation}\label{eq:gfoi}
\tGDD(x_0,y) \approx \tGDD(z_0,y) \ge cG(z_0,y) \approx G(x_0,y)\,.
\end{equation}
$\tG(\cdot,z_0)$ and $\tGDD(\cdot,y)$ are regular
$L$-harmonic in $D\cap B(Q,\rho /3)$ because $|z_0-Q|>5\rho/12$ and $|y-Q|>\rho/2$. 
By Lemma~\ref{BHPforL}, part I and
Lemma~\ref{BHP},
$$
\frac{\tGDD(x,y)}{\tGDD(x_0,y)} \approx
\frac{\tGDD(x,z_0)}{\tGDD(x_0,z_0)} \approx \frac{\GDD(x,z_0)}{\GDD(x_0,z_0)}
\approx \frac{\GDD(x,y)}{\GDD(x_0,y)}\,.
$$
By this and (\ref{eq:gfoi}),  $\tGDD(x,y) \ge
cG(x,y)$.\\

Step 2. We next prove the upper bound in (\ref{eq:egf}) for $\deltaDD(x) \ge \rho/4$. 
By part I,
$$
c_1^{-1} G(x,z) \le \tG(x,z) \le c_1 G(x,z)\,, \quad z \in D \setminus D^r\,.
$$
The constant $\CIV$ and other constants in what follows will only depend on $d$, $\alpha$, (the suprema in the Kato condition for) $b$, $r_0(D)$ and ${\rm diam}(D)$.

By (\ref{3GIntegral}) and Lemma~\ref{Lem3GIntegral},
$$
\int_{D} \frac{G(x,z)G(z,y)}{\deltaDD(z)\land |y-z|}
|b(z)|\,dz \le C_1 G(x,y)\,.
$$
Therefore by (\ref{eq:0}), 
\begin{align}
  \tG(x,y) &\le G(x,y) + c_1 d\int_{D \setminus D^r} \frac{G(x,z)G(z,y)}{\deltaDD(z)\land |y-z|} |b(z)|\,dz \nonumber\\
  &+ d \int_{D^r} \frac{\tG(x,z)G(z,y)}{\deltaDD(z)\land |y-z|} |b(z)|\,dz \nonumber\\
  & \le A G(x,y) + d\int_{D^r} \frac{\tG(x,z)G(z,y)}{|y-z| \land
    \deltaDD(z)} |b(z)|\,dz\,, \label{eq:toi}
\end{align}
where $A = 1+c_1 d C_1$. By Lemma~\ref{lem:GLDupper} and
(\ref{eq:1}) we obtain
\begin{align}
  \tG(x,y) &\le AG(x,y) + C_0 d\int_{D^r} \frac{|x-z|^{\alpha-d}G(z,y)}{\deltaDD(z)\land |y-z|} |b(z)|\,dz \nonumber\\
  & \le AG(x,y) + B(x)\,, \label{eq:3}
\end{align}
where $B(x)=\eta C_0 \delta_{D^r}(x)^{\alpha-d}$.  
We claim that for $n=0,1,\ldots$,
\begin{equation}\label{eq:AB}
\tG(x,y) \le A\big(1 + \eta +\cdots + \eta^n \big)G(x,y) +
\eta^n B(x)\,.
\end{equation}
This is proved by induction: we plug (\ref{eq:AB}) into
(\ref{eq:toi}), and use (\ref{eq:1}) and (\ref{eq:2}).
In consequence,
\begin{equation}\label{eq:ubi}
\tG(x,y) \le \frac{A}{1- \eta} G(x,y)\,.
\end{equation}
Step 3. We will now prove the upper bound in (\ref{eq:egf}) when $\deltaDD(x)< \rho/4$. 
We will consider $F_1$,  $F_3$ and $F$ from Step 1. 
If $x \in F_3\setminus F_1$ than we use the
same argument as in f), but this time all the terms in (\ref{eq:gfoi}) are comparable because of (\ref{eq:ubi}), and we obtain (\ref{eq:egf}). 
Finally, for $x \in F_1\subset F$ we have
$$
  \tGDD(x,y) =  
   \tG_F(x,y) + \int_{D \setminus F} \tP_F(x,z) \tGDD(z,y) \,dz\,.
$$
By Lemma~\ref{Theorem1s} , $\tG_F(x,y) \approx G_F(x,y)$. 
We already know that  for $z \in D \setminus F_1$,
$\tGDD(z,y) \approx G(z,y)$, and $\tP_F(x,z)\approx P_F(x,z)$ by (\ref{PoissonComp}). Thus,
$$
\tGDD(x,y) \approx G_F(x,y) + \int_{D \setminus F} P_F(x,z)
 G(z,y) \,dz = G(x,y)\,.
$$
The proof of Theorem~\ref{Theorem1} is complete.
In passing we only note that
(\ref{eq:djpt}) and (\ref{eq:egf}), and \cite{MR1654824} or \cite[Theorem 22]{MR1991120} yield sharp estimates of the Poisson kernel:
\begin{equation}\label{eq:ePk}
\tilde P_D(x,y)\approx P_D(x,y) \approx \delta_D(x)^{\alpha/2}\delta_{D^c}(y)^{-\alpha/2}\left[1\lor \delta_{D^c}(y)\right]^{-\alpha/2}|y-x|^{-d}\,.
\end{equation}

\section{Appendix}
Let function $\phi$ be of class $C^{1,1}$, i.e.\ satisfy
\begin{equation}\label{eq:C11}
|\nabla\phi(\tx)-\nabla\phi(\ty)|\leq \eta|\tx-\ty|\,,\quad \tx,\ty\in \RR^{d-1}\,,
\end{equation}
for some $\eta<\infty$. Let $f(t)=\phi((1-t)\ty+t\tx)$, $t\in [0,1]$.
We have
\begin{align}
&|\phi(\tx)-\phi(\ty)-\nabla \phi(\ty)\cdot (\tx-\ty)|
= \left|\int_0^1 \big(f\rq{}(t)-f\rq{}(0)\big)dt\right|\nonumber\\
&=\left|\int_0^1 (\tx-\ty)\cdot \big(\nabla \phi((1-t)\ty+t\tx)-\nabla \phi(\ty)\big)\right|
\leq 
\eta |\tx-\ty|^2/2\,.\label{eq:C11f}
\end{align}
We will consider nonlinear transformations of $\Rd$ defined as follows,
\begin{equation}\label{eq:T}
Tx=(\tx,x_d+\phi(\tx))\,,
\quad 
T^{-1}x=(\tx,x_d-\phi(\tx))\,.
\end{equation}
\begin{proof}[Proof of Lemma~\ref{l:loc}]
For $x=(x_1,\ldots,x_{d-1},x_d)\in \Rd$ we let $\tx=(x_1,\ldots,x_{d-1})$, so that $x=(\tx,x_d)$. The halfspace $H=\{x\in \Rd:\,x_d>0\}$ is a $C^{1,1}$ domain at each scale $r>0$, and we can localize it at $0$ by 
$$K = \{x \in \RR^d \colon0<x_d<r/2,\ |\tilde{x}| < r/4 +\sqrt{(r/4)^2-(x_d-r/4)^2} \}\,.$$ 
Put differently, $K$ is defined by the conditions: 
$|\tx|<r/2$ and 
$$r/4-\sqrt{(r/4)^2-(|\tx|-r/4)^2}<x_d<r/4+\sqrt{(r/4)^2-(|\tx|-r/4)^2}\,.$$
We see that $K$ is $C^{1,1}$ at scale $r/4$. We consider the inner and outer balls of radius $r/4$ for $K$, tangent at $Q\in \partial K$. 
Let $$B=\{x\in \Rd: 2(\nn,x-Q)> |x-Q|^2\}$$ be either one of them. Here $\nn\in \Rd$ and $|\nn|=r/4$.

Let $D$ be $C^{1,1}$ at a scale $r>0$. It is well known that up to isometry $D$ locally coincides with the image of the halfspace $H$ by a transformation $T$ of the form (\ref{eq:T}), see \cite[Section 2]{MR2286038}. Namely, by possibly changing coordinates, we may assume that $0\in \partial D$, $\phi$ satisfies (\ref{eq:C11}), $\phi(0)=0$, $\nabla\phi(0)=0$, 
$$
|\phi(\tilde{x})|\leq \frac{|\tilde{x}|^2}{r}<\frac{r}{4}\ ,\quad \mbox{provided} \quad  |\tilde{x}|<r/2\,,
$$
 and
$$
\{x\in D:\, |\tilde{x}|<r/2,\;|x_d|<r\,\}=\{x\in \Rd:\; |\tilde{x}|<r/2,\;\phi(\tilde{x})<x_d<r\}\,.
$$
We also have $\eta\leq c/r$ in (\ref{eq:C11}), where $c\geq 1$ is an absolute constant.

We define $F=TK\subset D$. We see that $F$ locally coincides with $D$,
and 
$F\subset \{x:\,|\tilde{x}|<r/2,\;-r/4<x_d<3r/4\}$, hence   ${\rm diam}(F)<2r$.
We claim that $F$ is $C^{1,1}$. 
The claim will follow from considering the image of $B$ by $T$.
Let $\xi=\nabla\phi(Q)$. We note that $|\xi|\leq \eta |\tQ|\leq \eta r/2\leq c/2$. 
We define
$$x\mapsto Sx=T(Q+x)-TQ=(\tx,x_d+[\phi(\tQ+\tx)-\phi(\tQ)])\ ,$$ $S^{-1}x=(\tx,x_d-[\phi(\tQ+\tx)-\phi(\tQ)])$, $
L x=(\tx,x_d+\xi\cdot \tx)$, $L^{-1}x=(\tx,x_d-\xi\cdot\tx)$.
We will use $L^{-1}$ as a linear approximation of $S^{-1}$ at $x=0$.
Let $L^{-1*}$ be the transpose of $L^{-1}$. 
We have $|Lx|\leq |x|+|\xi||x|\leq |x|(2+c)/2$. The same is true of $L^{-1}$ and $L^{-1*}$.
We note that 
$$
TB- TQ =S(B-Q) = \{x\in \Rd:\ 2(\nn,S^{-1}x)> |S^{-1}x|^2\}\,,
$$ 
and $x=0$ is on the boundary of the set. Consider the ball $B'=\{x:\, 2\kappa\, (L^{-1*}\nn,x)> |x|^2\}$. Note that $TQ \in  \partial(B'+TQ)$. We will verify our claim on $F$ by proving that if $2\kappa= c^{-1}\wedge (1+3c/4)^{-2}$, then $B' \subset S(B-Q)$, or $B'+TQ \subset TB$. 
To this end we note that 
\begin{equation}\label{eq:dBp}
2\kappa(L^{-1*}\nn,x)>|x|^2
\end{equation}
implies that $|x|< 2\kappa|L^{-1*}\nn|\leq r/2 $.
For such (small) $x$, by (\ref{eq:C11f}), we obtain
\begin{equation}\label{eq:T-1x}
|S^{-1}x|\leq |x|+|\xi||x|+\eta|x|^2/2\leq |x|(1+3c/4)\,.
\end{equation}
Similarly,
$$
|S^{-1}x-L^{-1} x|=|\phi(\tQ+\tx)-\phi(\tx)-\xi\cdot \tx|\leq \eta|\tx|^2/2\,.
$$
Now, (\ref{eq:dBp}) yields $2(\nn,L^{-1}x)>|x|^2/\kappa$, hence
\begin{align*}
2(\nn,S^{-1}x)&>
2(\nn,S^{-1}x)-2(\nn,L^{-1}x)+\frac{1}{\kappa}|x|^2
\geq \frac{1}{\kappa}|x|^2-2|\nn|\eta|\tx|^2/2\\
&\geq \frac{1}{2\kappa}|x|^2\geq (2\kappa)^{-1}(1+3c/4)^{-2}
|S^{-1}x|^2\geq |S^{-1}x|^2\,.
\end{align*}
The proof is complete.
\end{proof}

\begin{proof}[Proof of Lemma~\ref{l:wkp}]
By \cite{2009-TJ-KS-jee}
for $s>0$, $x\in \Rd$ and $\phi \in C_c^\infty\big((0,\infty)\times \Rd\big)$,
\begin{equation}\label{tp_fs}
\int_s^\infty \int_{\Rd} \tp(u-s,x,z) \left(\partial_u + \Delta_z^{\alpha/2}+b(z)\cdot \nabla_z\right)\phi(u,z) dz\,du = - \phi(s,x)\,.
\end{equation}
We note that the above integral is absolutely convergent. Indeed,
$$|\left(\partial_u + \Delta^{\alpha/2}_z + b(z) \cdot \nabla_z\right) \phi(u,z)| \le c(1+|b(z)|)\,,$$ and if $\phi(u,z) =0$ for $u > M$, then by (\ref{ptxy_comp}) and the remark following (\ref{eq:Kc}),
\begin{align*}
&\int_s^\infty \int_{\Rd} \tp(u-s,x,z) \left|\left(\partial_u + \Delta^{\alpha/2} + b(z) \cdot \nabla_z\right) \phi(u,z)\right| \,dz\,du \\
& \le  c_1\int_s^M \int_{\RR^d} p(u-s,x,z)(1+|b(z)|) \,dz\,du \\
& \le c_2 \int_\Rd (|x-z|^{\alpha-d}\wedge |x-z|^{-\alpha-d})(1+|b(z)|) \,dz<\infty\,.
\end{align*}
Furthermore, for $x\in \Rd$, $s\in \RR$ by (\ref{eq:Hunt}) and  (\ref{tp_fs}) we obtain
\begin{align}
&\int_s^\infty \int_{D} \tp_D(u-s,x,z) \left(\partial_u + \Delta_z^{\alpha/2} + b(z) \cdot \nabla_z\right) \phi(u,z) \,dz\,du\nonumber\\
& = \int_s^\infty \int_{\RR^d} \tp(u-s,x,z) \left(\partial_u + \Delta_z^{\alpha/2} + b(z) \cdot \nabla_z\right) \phi(u,z) \,dz\,du\nonumber\\
& - \tEE^x \int_{s+\tau_D}^\infty \int_{\RR^d} \tp(u-s-\tau_D,X_{\tau_D},z) \left(\partial_u + \Delta_z^{\alpha/2} + b(z) \cdot \nabla_z\right) \phi(u,z) \,dz\,du\nonumber\\
& = -\phi(s,x) + \tEE^x \phi(s+\tau_D, X_ {\tau_D}) \label{eq:pdec}\\
&= -\phi(s,x)\,.\nonumber
\end{align}
\end{proof}

\begin{proof}[Proof of Lemma~\ref{lem:lsgp}]
Let $s=0$, $x\in D$, $\phi \in C_c^\infty((0,\infty)\times \Rd)$ and assume that ${\rm supp}\, \phi \in (0,\infty)\times (\overline{D})^c$. 
By (\ref{eq:pdec}) and (\ref{eq;defulpn}) we obtain
$$
\tEE^x \phi(\tau_D,X_{\tau_D})=\int_0^\infty\int_D \int_{D^c}\tp_D(u,x,y)  \phi(u,z) \nu(z-y) dzdydu\,.
$$
\end{proof}

\begin{proof}[Proof of Lemma~\ref{lem:tgi}]
Let $\varphi \in C_c^\infty(D)$ and let $\chi(u) \in C_c^\infty(\RR)$ be such that  $\chi(u)=1$ for $u\in (-1,1)$. For $n=1,2, \ldots$ we define $\phi_n(u,z) = \varphi(x)\chi(u/n)$, and we have
$$
\left|\left(\partial_u + \Delta^{\alpha/2} + b(z) \cdot \nabla_z\right) \phi_n(u,z)\right| \le c(1+|b(z)|)
$$ 
with $c$ independent of $n$. By (\ref{tp_D_fs}), Lemma~\ref{lem:GLDupper} and dominated convergence, 
\begin{align*}
&-\varphi(x) = \lim_{n \to \infty} -\phi_n(0,x) \\
& =\lim_{n \to \infty} \int_0^\infty \int_{D} \tp_D(u,x,z) \left(\partial_u + \Delta_z^{\alpha/2} + b(z) \cdot \nabla_z\right) \phi_n(u,z) \,dz\,du\\
& = \int_s^\infty \int_{D} \tp_D(u,x,z) \left(\Delta^{\alpha/2}\varphi(z) + b(z) \cdot \nabla\varphi(z)\right) \,dz\,du\\
& = \int_{D} \tG_D(x,z) \left(\Delta^{\alpha/2}\varphi(z) + b(z) \cdot \nabla\varphi(z)\right)\,dz\,.
\end{align*}
\end{proof}

{\bf Acknowledgements.}  The results were presented at the conference
Nonlocal Operators and Partial Differential Equations, June 27-July 2, B\c{e}dlewo and The Sixth International Conference on L\'evy Processes: Theory and Applications, July 26-30, 2010, Dresden.
We thank the organizers for the invitation. In Dresden Professor Renming Song announced related sharp estimates of heat kernel $\tp_D$ for bounded $C^{1,1}$ open sets, analogous to \cite{CKS2008}.
{\bf Added in Proof.} The above mentioned results are available on arXiv (\cite{2010arXiv1011.3273C}). 
\bibliographystyle{abbrv} 
\bibliography{gradientpert}

\end{document}